\documentclass[12pt]{amsart}

\usepackage[dvipsnames]{xcolor}
\usepackage{mathtools}
\usepackage{graphicx}
\usepackage{amsmath}
\usepackage{color}
\usepackage{caption}
\usepackage{subcaption}
\usepackage{mwe}
\usepackage{amsfonts,amssymb,amscd}
\usepackage{mathrsfs}

\calclayout

\newtheoremstyle{remboldstyle}
  {}{}{\itshape}{}{\bfseries}{.}{.5em}{{\thmname{#1 }}{\thmnumber{#2}}{\thmnote{ (#3)}}}
\theoremstyle{remboldstyle}

\usepackage[outdir=./]{epstopdf}

\newtheorem{thm}{Theorem}[section]
\newtheorem{prop}[thm]{Proposition}
\newtheorem{lem}[thm]{Lemma}
\newtheorem{cor}[thm]{Corollary}

\newtheorem{question}[thm]{Question}

\theoremstyle{definition}
\newtheorem{definition}[thm]{Definition}

\newtheorem{rem}[thm]{Remark}
\newtheorem{notation}[thm]{Notation}

\newtheorem{thmx}{Theorem}

\numberwithin{equation}{section}

\def\Bbb{\mathbb}

\def\complex{\Bbb C}

\newcommand{\Chat}{\widehat{\mathbb{C}}}

\DeclareMathOperator{\dist}{dist} 

\begin{document}


\title[Analytic and Topological Nets]{Analytic and Topological Nets}


\author{Kirill Lazebnik}
\address{Kirill Lazebnik\\
Mathematics Department \\
University of North Texas \\
Denton, TX, 76205}
\email{Kirill.Lazebnik@unt.edu}





\begin{abstract} We characterize which planar graphs arise as the pullback, under a rational map $r$, of an analytic Jordan curve passing through the critical values of $r$. We also prove that such pullbacks are dense within the collection of $f^{-1}(\Sigma)$, where $f$ is a branched cover of the sphere and $\Sigma$ is a Jordan curve passing through the branched values of $f$. 

\end{abstract}


\maketitle


\section{Introduction}\label{intro}




In this paper we study the following objects:

\begin{definition}\label{analytic_net_defn} An \emph{analytic net} is a set of the form $r^{-1}(\Gamma)$ where $r: \Chat \rightarrow \Chat$ is a rational function and $\Gamma$ is an analytic Jordan curve passing through all the critical values of $r$. 
\end{definition}

Two examples are shown in Figure \ref{two_beginning_ex}, and several more in the Appendix. It was proven in \cite{MR1888795} that a real rational function $r$ with real critical points is determined (up to a real Mobius transformation) by its set of critical points, denoted $\textrm{CP}(r)$, together with the topology of the net $r^{-1}(\mathbb{R})$. Thus, nets are a natural object to look at in trying  to answer the following question of  \cite{38274}:


\begin{question}\label{thurston_question} Given a set of $2d-2$ points on $\Chat$ to be critical points (in the domain), it has been known since Schubert that there are Catalan$(d)$ rational functions [up to Mobius transformations in the range] with those critical points. Is there a conceptual way to describe and identify them?
\end{question}

In studying Question \ref{thurston_question}, it is useful to pass to the less rigid setting of \emph{branched covers} (these are the topological or ``floppy'' versions of rational functions: see the appendix for a definition), and the topological version of Definition \ref{analytic_net_defn}:





\begin{definition} A \emph{topological net} is a set of the form $f^{-1}(\Sigma)$ where $f: \Chat \rightarrow \Chat$ is a branched cover and $\Sigma$ is a Jordan curve passing through all the critical (branched) values of $f$. 
\end{definition}

\begin{figure}
\centering
\captionsetup{width=.9\linewidth}
\begin{subfigure}{.5\textwidth}
  \centering
  \includegraphics[width=.9\linewidth]{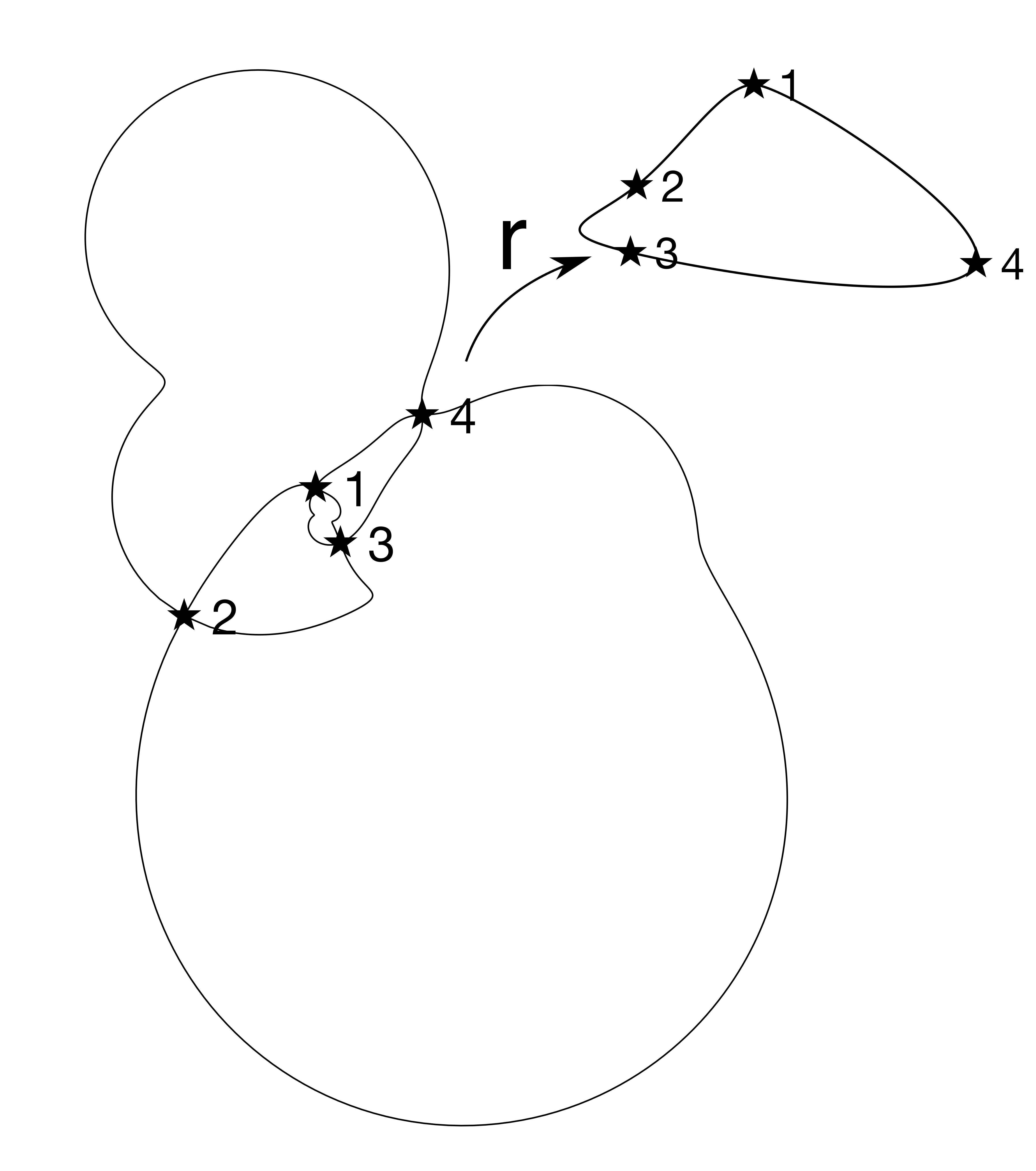}
  \caption{ }
  \label{fig:sub1}
\end{subfigure}%
\begin{subfigure}{.5\textwidth}
  \centering
  \includegraphics[width=.7\linewidth]{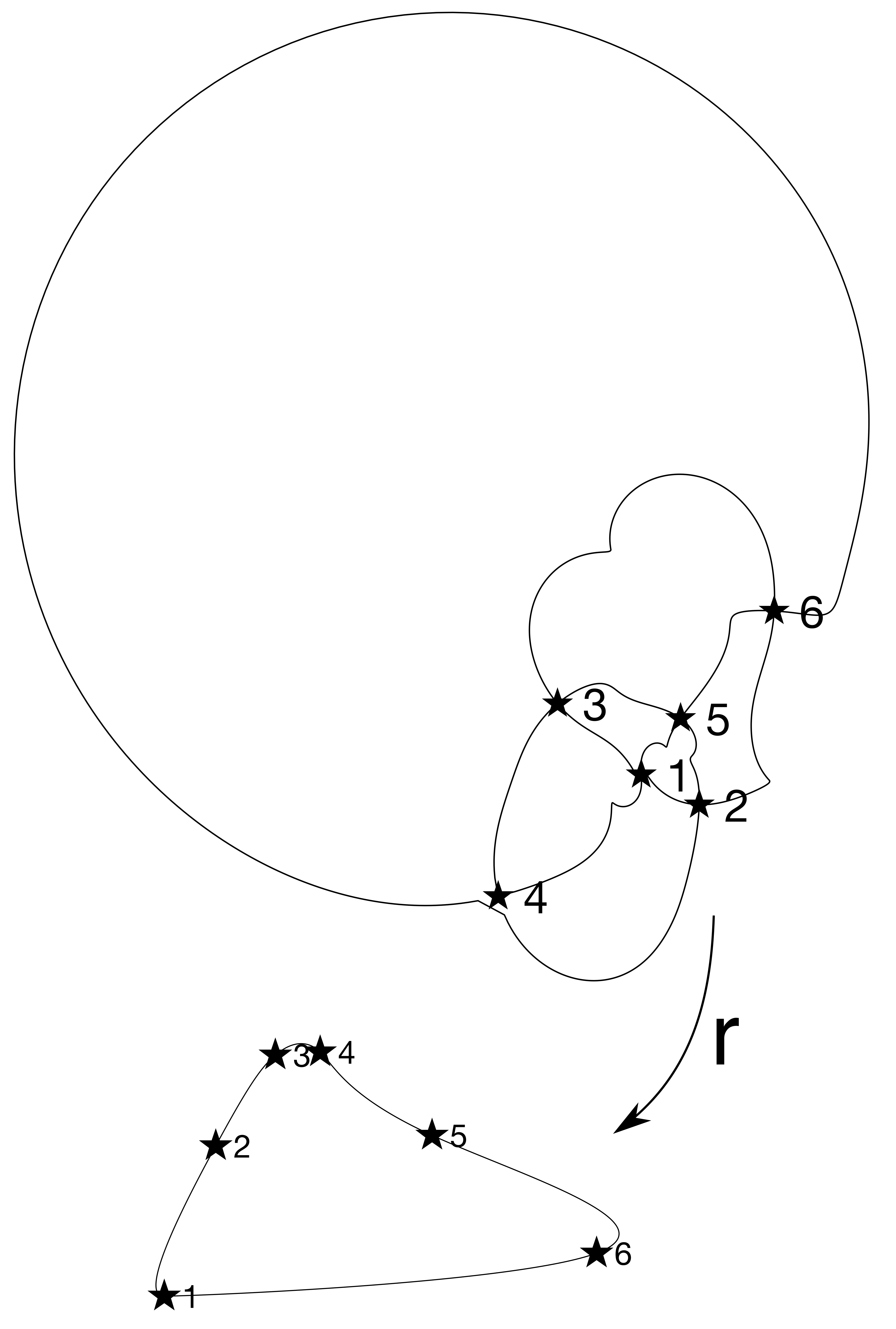}
  \caption{ }
  \label{fig:sub2}
\end{subfigure}
\caption{In each of (A) and (B) is pictured an analytic Jordan curve $\Gamma$ passing through the critical values of a rational map $r$, and the analytic net $r^{-1}(\Gamma)$. The critical points and values are marked with stars, and are numbered so that the $k^{\textrm{th}}$ critical point maps to the $k^{\textrm{th}}$ critical value. The example in (A) is of degree $3$, while the example in (B) is of degree $4$. As noted in Remark \ref{net_represents_map}, each rational map can be simply described as a conformal map of each component of $\Chat\setminus r^{-1}(\Gamma)$ onto a component of $\Chat\setminus\Gamma$, such that this piecewise definition extends continuously across the net. The nets and curves were produced in MATLAB.}
\label{two_beginning_ex}
\end{figure}

\begin{rem}\label{net_represents_map} A topological or analytic net affords a clear description of the underlying map: on each component of the complement of the net, the map is simply a homeomorphism onto one of the two components of the complement of the curve passing through the critical values (see Figure \ref{two_beginning_ex}).
\end{rem}

A net can naturally be viewed as a graph by placing vertices at each critical (branched) point, and in the generic case that each critical point is simple, each vertex of the graph has degree $4$. Thus, in light of Question \ref{thurston_question} and Remark \ref{net_represents_map}, one would like to know: 


\begin{question}\label{thuquest2} Which $4$-valent graphs in $\Chat$ are topological nets? Which analytic $4$-valent graphs in $\Chat$ are analytic nets?
\end{question}

\noindent Question \ref{thuquest2} has two parts, the topological part (the first half) and the analytic part (the second half). The topological part was answered in \cite{MR4205641} (for the non-generic setting, see \cite{2023arXiv230407207L}). In this paper we answer the analytic part.


Before stating our result precisely, we remark that while every analytic net is a topological net, not every topological net with analytic edges is an analytic net (see Remark \ref{counterexample}). This leads to the question: how does the space of (rigid) analytic nets fit within the larger space of (flexible) topological nets? We give one answer in the following theorem, which asserts density with respect to the Hausdorff metric $d_H(\cdot, \cdot)$.  We denote $N_{\varepsilon}(E):=\{z \in \Chat : \dist(z, E)<\varepsilon\}$ for $\varepsilon>0$ and $E\subset\Chat$, and we denote the critical (branched) points of a branched cover $f$ by $\textrm{CP}(f)$. 






\begin{thmx} \label{density_thm} \emph{For any topological net $f^{-1}(\Sigma)$ and $\varepsilon>0$, there exists an analytic net $r^{-1}(\Gamma)$ so that $d_H(f^{-1}(\Sigma), r^{-1}(\Gamma))<\varepsilon$ and $\emph{CP}(f)\subset N_{\varepsilon}(\emph{CP}(r))$.  }
\end{thmx}

Since a net essentially determines the underlying map, Theorem \ref{density_thm} may be interpreted as saying that every branched cover has a rational function which is ``close'' to it, in the sense of Theorem \ref{density_thm}, although we remark the degree of $r$ in Theorem \ref{density_thm} will in general be much larger than the degree of $f$. We also remark that the related question of which branched covers are ``equivalent'' to rational maps is well studied and has generated a large literature (see \cite{MR1251582}, \cite{MR4125449}).




We will now describe our answer to the analytic part of Question \ref{thuquest2}. First we need some notation and definitions to state our result (Theorem \ref{char_thm} below) precisely.

\begin{notation} Given a Jordan domain $\Omega$ and a point $z_0\in\Omega$, we denote the harmonic measure of a Borel subset $I\subset\partial\Omega$ by $\omega(I, z_0, \Omega)$. If $\partial\Omega$ is rectifiable, then harmonic measure $\omega(\cdot, z, \Omega)$ and Lebesgue measure on $\partial\Omega$  are mutually absolutely continuous (see Theorem IV.1.2 of \cite{MR2450237}), and we will denote the Radon-Nikodym derivative of $\omega(\cdot, z, \Omega)$ with respect to Lebesgue measure at a point $x\in\partial\Omega$ by $d\omega(\cdot, z, \Omega)(x)$. 
\end{notation}





\begin{definition}\label{harmonic_parametrization} For any $4$-valent graph $G\subset\Chat$, we denote by $V(G)$ the set of vertices of $G$. We say a $4$-valent graph $G$ is \emph{analytic} if each edge of $G$ is a strict subset of an analytic curve, and the four angles of intersection at any vertex are all $\pi/2$. 
 A \emph{marking} of $G$ is a pair $(\mathcal{Z}, X)$, where $\mathcal{Z}$ is a set of points one in each face of $G$, and $X\subset G\setminus V(G)$ satisfies $|X\cap\partial F|=1$ for each face $F$ of $G$. We use the notation $z_F:=\mathcal{Z}\cap F$, and $x_F:=X\cap\partial F$ for any face $F$. 
\end{definition}


\begin{definition} Given a marking $(\mathcal{Z}, X)$ of an analytic graph $G$, we define a function $g_F$ on the boundary of each face $F$ of $G$ as follows. First, let $\gamma: [0,1] \rightarrow \partial F$ be a parametrization (counter-clockwise about $F$) satisfying $\gamma(0)=x_F$ and $\omega(\gamma([0,t]), z_F, F)=t$ for all $0\leq t \leq1$. For  $\gamma(t)\not\in V(G)$, we denote by $F_t$ the unique face neighboring $F$ with $\gamma(t)\in\partial F_t$. We set:
\begin{equation}\label{goverboundF} g_F(t):=\frac{d\omega(\cdot, z_{F_t}, F_t)(\gamma(t))}{d\omega(\cdot, z_F, F)(\gamma(t))} \textrm{ for } t\in[0,1]\setminus\{ t : \gamma(t)\in V(G) \}. \end{equation}
We denote the domain $[0,1]\setminus\{ t : \gamma(t)\in V(G) \}$ of $g_F$ by $\textrm{dom}(g_F)$.
\end{definition} 

\begin{rem} We use the notation $\mathcal{O}(\mathbb{T})$ for the class of analytic self-homeomorphisms of $\mathbb{T}:=\{z : |z|=1\}$; namely the self-homeomorphisms of $\mathbb{T}$ which extend to holomorphic functions in some neighborhood of $\mathbb{T}$. We also remark that any $4$-valent graph can always be $2$-colored, that is: each face can be assigned one of two colors such that any two faces which share an edge have different colors. Moreover there are exactly two $2$-colorings of any $4$-valent graph. Our convention will be to use black and white for the $2$ colors. 
\end{rem}

Our answer to the analytic part of Question \ref{thuquest2} is the following, which roughly says an analytic graph is an analytic net if and only if there exists a choice of base points for harmonic measure in each face of $G$ so that the ratio of harmonic measures from either side of the boundary of a face is the same over all faces:

\begin{thmx}\label{char_thm}  \emph{ Let $G$ be a $2$-colored, analytic $4$-valent graph. Then $G$ is an analytic net if and only if there exists a marking $(\mathcal{Z}, X)$ of $G$ and $\kappa\in \mathcal{O}(\mathbb{T})$ so that for each white face $F$ we have $g_F(t)=|\kappa'(e^{2\pi it})|$ for all $t\in\emph{dom}(g_F)$.  }
\end{thmx}


\begin{rem}\label{counterexample} Recall that harmonic measure $\omega(I, z_0, \Omega)$ coincides with the probability that a Brownian motion started at $z_0$ exits $\Omega$ through $I$. With this in mind, Theorem \ref{char_thm} can be used to help identify which analytic $4$-valent graphs are analytic nets (see Figure \ref{not_analytic_net} and its caption).
\end{rem}

We conclude the introduction by briefly describing the proofs of Theorems \ref{density_thm} and \ref{char_thm}, starting with Theorem \ref{char_thm}. If $G=r^{-1}(\Gamma)$ is an analytic net, we prove that $\kappa$ is the conformal welding (see Definition \ref{conf_welding_defn}) associated to the curve $\Gamma$, and we may take $\mathcal{Z}$ to be the preimage (under $r$) of two points in different components of $\Chat\setminus\Gamma$, and $X$ to be the preimage (under $r$) of a non-critical value on $\Gamma$. Conversely, suppose we are given $G$, $(\mathcal{Z}, X)$, $\kappa$ satisfying the conditions in Theorem \ref{char_thm}. The Measurable Riemann Mapping Theorem (henceforth abbreviated MRMT) implies the existence of an analytic Jordan curve $\Gamma$ such that $\kappa$ is the derivative of a conformal welding associated to $\Gamma$. We then prove that the hypotheses on $G$, $(\mathcal{Z}, X)$, $\kappa$ imply that the map defined by mapping conformally each face of $G$ to one of the two components of $\Chat\setminus\Gamma$ extends continuously across $G$ and hence is rational.  

Now we turn to a description of the proof of Theorem \ref{density_thm}. Given the graph $G:=f^{-1}(\Sigma)$, we define a map $g$ in the union of the faces of $G$ as follows. In each white face of $G$, $g$ is a conformal map to $\mathbb{D}$ followed by $z\mapsto z^n+\delta z$ for large $n$ and small $\delta$, and in each black face the map $g$ is a conformal map to $\mathbb{D}^*:=\Chat\setminus\overline{\mathbb{D}}$ followed by $z\mapsto z^m+\delta z$ for $m\gg n$. There is a curve $\Gamma$ running through the critical values of $g$ so that $g^{-1}(\Gamma)\approx G$, however $g$ does not extend continuously across $G$. We prove that there is an arbitrarily small neighborhood of $G$ in which we can quasiregularly interpolate between the definitions of $g$ in different faces, with dilatation bounded independently of how small the neighborhood of $G$ is. Thus the MRMT implies the existence of a quasiconformal $\phi:\Chat\rightarrow\Chat$ so that $r:=g\circ\phi^{-1}$ is rational and $r\approx g$, and hence $r^{-1}(\Gamma)\approx g^{-1}(\Gamma)\approx G$. The details of the interpolation are rather technical and rely in part on some existing results, including several lemmas from \cite{Bis15} and \cite{MR4041106}.

\begin{figure}[ht!]
{\includegraphics[width=.7\textwidth]{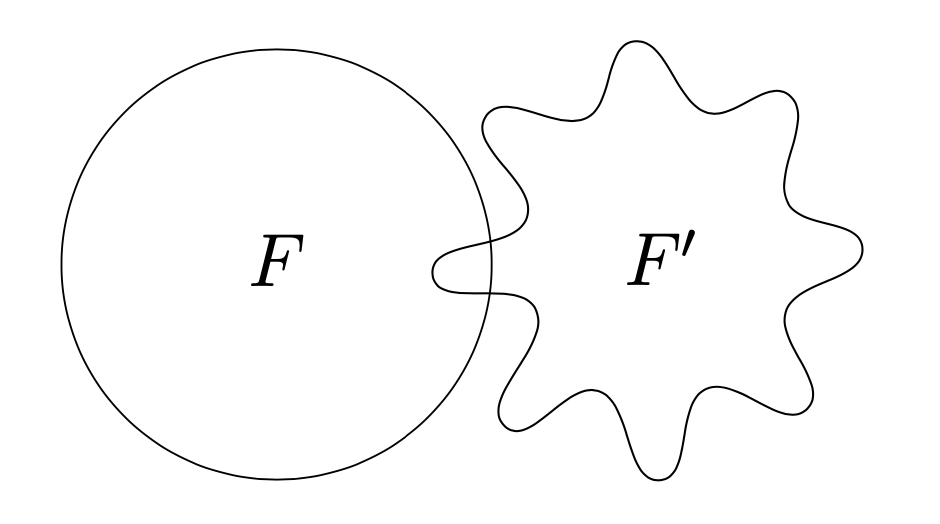}}
	\caption{Pictured is a sketch of a topological net with analytic edges. Theorem \ref{char_thm} implies that this topological net is not an analytic net since the function $g_{F'}$ will have more local extrema than the function $g_F$.  }
\label{fig:not_analytic_net}
\end{figure}

\section{Proof of Theorem \ref{char_thm}}

\begin{definition}\label{conf_welding_defn} Let $\Gamma\subset\Chat$ be a Jordan curve, and let $\phi_1$, $\phi_2$ denote Riemann mappings of $\mathbb{D}$, $\mathbb{D}^*$ onto the two components of $\Chat\setminus\Gamma$. The mappings $\phi_1$, $\phi_2$ both extend to $\mathbb{T}$, and the homeomorphism $\phi_2^{-1}\circ\phi_1: \mathbb{T} \rightarrow \mathbb{T}$ is called a \emph{conformal welding}. We denote the class of all such homeomorphisms (obtained by considering all Jordan curves) by $\textrm{Weld}$. We set $\textrm{Weld}_A$ to be the class of conformal weldings obtained by considering only analytic Jordan curves $\Gamma$.
\end{definition} 

\begin{prop}\label{confweldingequiv} $\mathcal{O}(\mathbb{T})=\textrm{Weld}_A$. 
\end{prop}

\begin{proof} Let $\phi_2^{-1}\circ\phi_1\in\textrm{Weld}_A$ be a conformal welding associated to an analytic Jordan curve $\Gamma$. Both $\phi_1$, $\phi_2$ extend analytically across $\mathbb{T}$ by the Schwarz Reflection principle for analytic arcs, and hence $\phi_2^{-1}\circ\phi_1\in\mathcal{O}(\mathbb{T})$. 

Conversely, given $\kappa\in\mathcal{O}(\mathbb{T})$, $\kappa$ may extended to a $K$-quasiconformal mapping $\kappa: \mathbb{D}\rightarrow\mathbb{D}$. Consider the Beltrami coefficient $\mu$ defined by $\mu(z):=\kappa_{\overline{z}}(z)/\kappa_z(z)$ for $z\in\mathbb{D}$, and $\mu(z)=0$ for $z\not\in\mathbb{D}$. The MRMT implies there exists a $K$-quasiconformal homeomorphism $\phi: \Chat\rightarrow\Chat$ solving the equation $\phi_{\overline{z}}(z)=\mu(z)\cdot\phi_{z}(z)$ a.e.. Let $\Sigma:=\phi(\mathbb{T})$. The identity
\[ \kappa=\kappa\circ\phi^{-1}\circ\phi \] 
shows that $\kappa$ is a conformal welding associated to $\Sigma$, since $\phi: \mathbb{D}^* \rightarrow \phi(\mathbb{D}^*)$ is conformal (since $\mu\equiv0$ in $\mathbb{D}^*$), and $\kappa\circ\phi^{-1}=(\phi\circ\kappa^{-1})^{-1}$ with $\phi\circ\kappa^{-1}: \mathbb{D} \rightarrow \phi(\mathbb{D})$ conformal (since $\mu(z):=\kappa_{\overline{z}}(z)/\kappa_z(z)$ for $z\in\mathbb{D}$). Thus $\kappa$ is a conformal welding, and so it remains to show that the curve $\Sigma$ is analytic, which follows since $\kappa$ is holomorphic in a neighborhood of $\mathbb{T}$, and hence so is $\phi$. 
\end{proof}

\noindent Thus, by Proposition \ref{confweldingequiv}, Theorem \ref{char_thm} is equivalent to the following: 

\begin{thm}\label{char_thm_mod} Let $G$ be an analytic $4$-valent graph. Then $G$ is an analytic net if and only if there exists a marking $(\mathcal{Z}, X)$ of $G$ and $\kappa\in\textrm{Weld}_A$ so that for each white face $F$ we have $g_F(t)=|\kappa'(e^{2\pi it})|$ for all $t\in\emph{dom}(g_F)$. 
\end{thm}

 We will break Theorem \ref{char_thm_mod} into two parts: necessity (Theorem \ref{first_half_of_thm}) and sufficiency (Theorem \ref{sec_half_of_thm}). The proofs of both parts will use the following fact (Proposition \ref{form_of_deriv})

\begin{prop}\label{form_of_deriv} Let $\Omega$ be a Jordan domain with piecewise-analytic boundary, and $\Omega'$ a Jordan domain with analytic boundary. Then, for any conformal map $\phi: \Omega \rightarrow \Omega'$ and $z\in\Omega$, we have
\begin{equation} d\omega(\cdot, z, \Omega)(\zeta)=|\phi'(\zeta)|\cdot d\omega(\cdot, \phi(z), \Omega')(\phi(\zeta))
\end{equation}
for all smooth points $\zeta\in\partial\Omega$.
\end{prop}

\begin{proof} This follows from conformal invariance of harmonic measure.
\end{proof}

\begin{notation} If $\Sigma\subset\Chat$ is an analytic Jordan curve such that $0$, $\infty$ lie in different components of $\Chat\setminus\Sigma$, we will denote by $\textrm{int}(\Sigma)$ (resp. $\textrm{ext}(\Sigma)$) the component of $\Chat\setminus\Sigma$ containing $0$ (resp. $\infty$).
\end{notation}

\noindent It will be useful to isolate the following special case of Proposition \ref{form_of_deriv}.

\begin{cor}\label{cor_form_of_deriv} Let $\Sigma\subset\Chat$ be an analytic Jordan curve such that $0$, $\infty$ lie in different components of $\Chat\setminus\Sigma$, and let $\phi_0$ (resp. $\phi_\infty$) be a conformal map of $\emph{int}(\Sigma)$ (resp. $\emph{ext}(\Sigma)$) onto $\mathbb{D}$ (resp. $\mathbb{D}^*$) fixing $0$ (resp. $\infty$). Then \begin{equation}\label{first1} d\omega(\cdot, 0, \emph{int}(\Sigma))(\zeta)=|\phi_0'(\zeta)|/2\pi,
\end{equation}
\begin{equation}\label{first2} d\omega(\cdot, \infty, \emph{ext}(\Sigma))(\zeta)=|\phi_\infty'(\zeta)|/2\pi,
\end{equation}
for all $\zeta\in\Sigma$. 
\end{cor}

\begin{proof} This follows from Proposition \ref{form_of_deriv} together with the fact that the harmonic measures $\omega(\cdot, \mathbb{D}, 0)$, $\omega(\cdot, \mathbb{D}^*, \infty)$ on $\mathbb{T}$ both coincide with length measure on $\mathbb{T}$.
\end{proof}

\begin{thm}\label{first_half_of_thm} Let $G$ be an analytic $4$-valent graph. If $G$ is an analytic net, then there exists a marking $(\mathcal{Z}, X)$ of $G$ and $\kappa\in\textrm{Weld}_A$ so that for each white face $F$ we have $g_F(t)=|\kappa'(e^{2\pi it})|$ for all $t\in\emph{dom}(g_F)$. 
\end{thm}

\begin{proof} Suppose $G$ is an analytic net, in other words there exists a rational map $r$ and an analytic Jordan curve $\Sigma$ passing through $\textrm{CV}(r)$ such that $r^{-1}(\Sigma)=G$. Take two points, one in each component of $\Chat\setminus\Sigma$: we may assume without loss of generality that these two points are $0$, $\infty$ and that black (resp. white) components are mapped to the component of $\Chat\setminus\Sigma$ containing $\infty$ (resp. $0$). Set $\mathcal{Z}:=r^{-1}(\{0, \infty\})$ and $X:=r^{-1}(y)$ where $y\in\Sigma\setminus\textrm{CV}(r)$. Let $\phi_0$, (resp. $\phi_\infty$) be a conformal mapping of $\textrm{int}(\Sigma)$, (resp. $\textrm{ext}(\Sigma)$) onto $\mathbb{D}$ (resp. $\mathbb{D}^*$) fixing $0$ (resp. $\infty$), so that $\kappa:=\phi_\infty\circ\phi_0^{-1}$ is a conformal welding. We normalize $\phi_0(y)=\phi_\infty(y)=1$. Fix a white face $F$. We will show that $g_F(t)=|\kappa'(e^{2\pi it})|$ for all $t\in\textrm{dom}(g_F)$.

Note that the map $r$ is injective on each face of $G$. Thus, we deduce: 
\begin{equation}\label{gexp} g(t):=\frac{d\omega(\cdot, z_{F_t}, F_t)(\gamma(t))}{d\omega(\cdot, z_F, F)(\gamma(t))}=\frac{d\omega(\cdot, \infty, \textrm{ext}(\Sigma))(r\circ\gamma(t))}{d\omega(\cdot, 0, \textrm{int}(\Sigma))(r\circ\gamma(t))}=\left|\frac{\phi_\infty'(r\circ\gamma(t))}{\phi_0'(r\circ\gamma(t))}\right|,
\end{equation} 
for all $t\in\textrm{dom}(g_F)$, where the first $=$ follows from Proposition \ref{form_of_deriv}, and the second $=$ follows from Corollary \ref{cor_form_of_deriv}.
By the chain rule we have that
\begin{equation}\label{frac2}\frac{\phi_\infty'(r\circ\gamma(t))}{\phi_0'(r\circ\gamma(t))}=\left( \phi_\infty\circ\phi_0^{-1} \right)'\left( \phi_0\circ r\circ \gamma(t) \right).
\end{equation}
Recall the normalizations $\phi_0\circ r(x_F)=1$ and $\phi_0\circ r(z_F)=0$, and recall also that $\gamma(t)$ was defined so that $\omega(\gamma([0, t]), z_F, F)=t$. Thus we deduce by conformal invariance of harmonic measure that $\phi_0(r\circ\gamma(t))=\exp(2\pi it)$. Together with (\ref{gexp}) and (\ref{frac2}), this implies that 
\begin{equation}\label{finalg} g(t)=|\left( \phi_\infty\circ\phi_0^{-1} \right)'\left(\exp(2\pi it)\right)|=:|\kappa'(e^{2\pi it})|
\end{equation}
for all $t\in\textrm{dom}(g_F)$.
\end{proof}

\begin{thm}\label{sec_half_of_thm} Let $G$ be an analytic $4$-valent graph. Suppose there exists a marking $(\mathcal{Z}, X)$ of $G$ and $\kappa\in\textrm{Weld}_A$ so that for each white face $F$ we have $g_F(t)=|\kappa'(e^{2\pi it})|$ for all $t\in\emph{dom}(g_F)$. Then $G$ is an analytic net.
\end{thm}

\begin{proof} Let $G$, $\mathcal{Z}$, $X$, $\kappa$ be as in the statement, and let $\Sigma$ be the analytic Jordan curve giving rise to the conformal welding $\kappa:=\phi_\infty\circ\phi_0^{-1}$ where $\phi_\infty$ (resp. $\phi_0$) is a conformal mapping of $\textrm{ext}(\Sigma)$ (resp. $\textrm{int}(\Sigma)$) onto $\mathbb{D}^*$ (resp. $\mathbb{D}$). We may assume without loss of generality that $0$, $\infty$ lie in different components of $\Chat\setminus\Sigma$, and $\phi_0(0)=0$, $\phi_\infty(\infty)=\infty$, $\phi_\infty\circ\phi_0^{-1}(1)=1$. Define a mapping $r: \Chat\setminus G\rightarrow \Chat\setminus\Sigma$ as follows. In each white (resp. black) face $F$, set $r$ to be the conformal mapping from $F$ to $\textrm{int}(\Sigma)$ (resp. $\textrm{ext}(\Sigma)$), normalized so that $r(z_F)=0$ (resp. $r(z_F)=\infty$) and $r(x_F)=\phi_0^{-1}(1)$ (resp. $r(x_F)=\phi_\infty^{-1}(1)$). By removability of analytic arcs for holomorphic mappings, the proof will be finished once we show that $r$ extends continuously across $G$. 

Let $F$ be a white face of $G$, and consider the parametrization $\gamma: [0,1] \rightarrow \partial F$ as in Definition \ref{harmonic_parametrization}. Denote by $r_i: \partial F \rightarrow \Sigma$ the boundary value extension of $r|_F$, and let $r_o: \partial F\setminus V(G) \rightarrow \Sigma$ denote the boundary value extension of $r$ restricted to the union of the black faces of $G$. Fix $t\in[0,1]$, and let 
\begin{equation}\label{defnofst} s_t:=\omega(r_i\circ\gamma([0, t]), \infty, \textrm{ext}(\Sigma)).
\end{equation}
 Then, since $r_i\circ\gamma|_{[0,t]}$ parametrizes $r_i\circ\gamma([0, t])$, we have
\begin{equation} s_t=\int_0^t d\omega(\cdot, \infty, \textrm{ext}(\Sigma))(r_i\circ\gamma(x))\cdot|(r_i\circ\gamma)'(x)|dx.
 \end{equation}
 Multiplying by $1$ gives us
 \begin{equation}\label{s_t2} s_t=\int_0^t d\omega(\cdot, 0, \textrm{int}(\Sigma))(r_i\circ\gamma(x))\frac{d\omega(\cdot, \infty, \textrm{ext}(\Sigma))(r_i\circ\gamma(x))}{d\omega(, 0, \textrm{int}(\Sigma))(r_i\circ\gamma(x))}\cdot|(r_i\circ\gamma)'(x)|dx.
 \end{equation}

Corollary \ref{cor_form_of_deriv} and the chain rule imply that
\begin{equation}\label{eqtn1} \frac{d\omega(\cdot, \infty, \textrm{ext}(\Sigma))(r_i\circ\gamma(x))}{d\omega(, 0, \textrm{int}(\Sigma))(r_i\circ\gamma(x))}=\left|\frac{\phi_\infty'(r_i\circ\gamma(x))}{\phi_0'(r_i\circ\gamma(x))}\right|=\left|\left( \phi_\infty\circ\phi_0^{-1} \right)'\left( \phi_0\circ r_i\circ \gamma(t) \right)\right|.
\end{equation}
The normalization $\phi_0\circ r_i(x_F)=1$ together with conformal invariance of harmonic measure implies that 
\begin{equation}\label{eqtn2} \phi_0\circ r_i\circ \gamma(t)=\exp(2\pi i t) \textrm{ for all } t\in[0,1], \end{equation}
and so (\ref{eqtn1}) and (\ref{eqtn2}) together imply
\begin{equation} \frac{d\omega(\cdot, \infty, \textrm{ext}(\Sigma))(r_i\circ\gamma(x))}{d\omega(, 0, \textrm{int}(\Sigma))(r_i\circ\gamma(x))}= |\kappa'(e^{2\pi it})| \textrm{ for all } t\in[0,1]. 
\end{equation}
Thus the assumption  $g_F(t)=|\kappa'(e^{2\pi it})|$ for all $t\in\textrm{dom}(g_F)$ together with (\ref{s_t2}) implies that
  \begin{equation}\label{anotherst} s_t=\int_0^t d\omega(\cdot, 0, \textrm{int}(\Sigma))(r_i\circ\gamma(x))\frac{d\omega(\cdot, z_{F_x}, F_x)(\gamma(x))}{d\omega(\cdot, z_F, F)(\gamma(x))}\cdot|(r_i\circ\gamma)'(x)|dx
 \end{equation}

Next, we observe that 
 \begin{equation}\label{domega} d\omega(\cdot, 0, \textrm{int}(\Sigma))(r_i\circ\gamma(x))|r_i'(\gamma(x))|=d\omega(\cdot, z_F, F)(\gamma(x))
 \end{equation}
 for $x\in[0,1]\setminus\{x : \gamma(x)\in V(G)\}$ by Proposition \ref{form_of_deriv}, and so (\ref{anotherst}) and (\ref{domega}) together with the chain rule imply that
 \begin{equation}\label{anotherst1} s_t=\int_0^t d\omega(\cdot, z_{F_x}, F_x)(\gamma(x))\cdot|\gamma'(x)|dx. 
 \end{equation}
 Next, Proposition \ref{form_of_deriv} and (\ref{anotherst1}) imply that
 \begin{equation} s_t=\int_0^t d\omega(\cdot, \infty, \textrm{ext}(\Sigma))(r_0\circ\gamma(x))\cdot|(r_0\circ\gamma)'(x)|dx,
 \end{equation} 
 which together with the definition (\ref{defnofst}) of $s_t$ implies that
   \begin{equation} \omega(r_i\circ\gamma([0, t]), \infty, \textrm{ext}(\Sigma))= \omega(r_o\circ\gamma([0, t]), \infty, \textrm{ext}(\Sigma)). 
 \end{equation}
 Since we normalized so that $r_i\circ\gamma(0)=r_o\circ\gamma(0)$, we conclude that $r_i(\gamma(t))=r_o(\gamma(t))$. Since $t\in[0,1]$ was arbitrary, we conclude that $r$ extends continuously across $\partial F$, and since $F$ was an arbitrary white face, we conclude that $r$ extends continuously across $G$. 
\end{proof}


\section{An Interpolation between $z\mapsto z^m$ and $z\mapsto z^m+\delta z$}\label{interpolation_section}

Having proven Theorem \ref{char_thm}, we now turn to the proof of Theorem \ref{density_thm}. In this Section we focus on a technical result we will need on the existence of an efficient interpolation between $z\mapsto z^m$ on $|z|=1$ and $z\mapsto z^m+\delta z$ on $r\mathbb{D}$, where $r<1$, $m\in\mathbb{N}$ and $\delta>0$.

\begin{definition} Consider the smooth bump function:
\[ b(x):=\begin{cases} 
      \exp(1+\frac{1}{x^2-1}) & \textrm{ if } 0\leq x < 1 \\
      0 & \textrm{ if } x\geq 1.
   \end{cases}
\]
We use the transformation $\phi_r(x):=\frac{x-r}{1-r}$ in order to define the modified smooth bump function: 
\[ \hat{\eta}_r(x):=\begin{cases} 
      1 & \textrm{ if }x\leq r \\
      b(\phi_r(x)) & \textrm{ if } r \leq x \leq 1 \\
      0 & \textrm{ if }x\geq 1,
   \end{cases}
\]
and we define $\eta_r(z):=\hat{\eta}_r\left(|z|\right)$. We set $\iota_{m, \delta}(z):= z^m + \delta z \cdot\eta_{r_{m, \delta}}(z)$ for $z\in\mathbb{D}$ with $r_{m, \delta}:= 1-(4\delta)/m$. 
\end{definition}


\noindent We refer to Section 3 of \cite{MR4041106} for a proof of the following lemma.

\begin{lem}\label{interpolation} There exist $n\in\mathbb{N}$, $\delta>0$, and  $k<1$ such that if $m>n$ and $\varepsilon\leq\delta$, then 
\begin{equation} r_{m, \varepsilon}>(\varepsilon/m)^{1/(m-1)} \textrm{ and } \left|\left|\frac{(\iota_{m, \varepsilon})_{\overline{z}}}{(\iota_{m, \varepsilon})_z}\right|\right|_{L^\infty(\mathbb{D})}<k.
\end{equation} 
\end{lem}

\begin{notation} We will use the notation $\iota_m:=\iota_{m, \delta}$ for $m$ and $\delta$ as in the conclusion of Lemma \ref{interpolation}, and we will sometimes omit the subscript $m$ and simply write $\iota$. 
\end{notation}


\noindent It will be important to record the critical points and values of $\iota$:
\begin{equation}\label{cpcviota} \textrm{CP}(\iota_m)=\left(\frac{-\delta}{m}\right)^{\frac{1}{m-1}}, \textrm{ and } \textrm{CV}(\iota_m)=\delta\left(\frac{-\delta}{m}\right)^{\frac{1}{m-1}}\left(\frac{m-1}{m}\right).
\end{equation}

\begin{rem} We extend $\iota$ to a holomorphic self-map of $\Chat$ by Schwarz-reflection: $\iota(z):=\overline{1/\iota(1/\overline{z})}$ for $z\in \mathbb{D}^*$. Thus, statements about $\iota|_{\mathbb{D}}$ easily translate to statements about $\iota|_{\mathbb{D}^*}$, for instance the critical points and values of $\iota|_{\mathbb{D}^*}$ are obtained by inverting the formulas in (\ref{cpcviota}). 
\end{rem}

\begin{notation} We will use the notation $A(r_1, r_2):=\{z\in\mathbb{C} : r_1 < |z| < r_2\}$.
\end{notation}

\begin{definition}\label{defnofcn} We set
\begin{equation} c_m:=\left(\frac{\delta}{m}\right)^{\frac{2}{m-1}}
\end{equation}
\end{definition}

\begin{lem}\label{annulus_pullback_lemma} The sequence $c_m\rightarrow1$ as $m\rightarrow\infty$, and if $|z|<c_m$, then
\begin{equation}\label{iota_bound} |\iota(z)| \leq c_m^m+\delta c_m < |\emph{CV}(\iota)|. \end{equation}
In particular, the annulus $A(c_m^m+\delta c_m, 1)$ contains the critical values of $\iota$, and for any $\varepsilon>0$ we have $\iota^{-1}(A(c_m^m+\delta c_m,1))\subset A(1-\varepsilon, 1)$ for all sufficiently large $m$. 
\end{lem}

\begin{proof} The conclusion $c_m\rightarrow1$ as $m\rightarrow\infty$ follows from L'H\"opital's rule. The statement $|\iota(z)| \leq c_m^m+\delta c_m$ is simply the triangle inequality. We calculate that
\begin{equation} c_m^m+\delta c_m=\delta\left(\delta/m\right)^{\frac{2}{m-1}}\left(\frac{\delta}{m^2}+1\right),
\end{equation}
and so, by (\ref{cpcviota}), the inequality 
\begin{equation} c_m^m+\delta c_m < |\textrm{CV}(\iota)| \end{equation} 
is equivalent to 
\begin{equation} \frac{\delta}{m}\left(\frac{\delta}{m^2}+1\right) < 1-\frac{1}{m} \end{equation}
which is true for large $m$ since the left-hand side tends to $0$ and the right-hand side to $1$ as $m\rightarrow\infty$. 
\end{proof}

\section{Proof of Theorem \ref{density_thm}}\label{proof_of_density_section} 

Having collected the relevant facts about the map $\iota$ in the previous section, we now turn to the proof of Theorem \ref{density_thm}. 

\begin{notation} Throughout this section, we will fix (as in the statement of Theorem \ref{density_thm}) a topological net $G:=f^{-1}(\Sigma)\subset\Chat$. For the purposes of proving Theorem \ref{density_thm}, after replacing $G$ by a Hausdorff approximant, we may assume that the edges of $G$ are analytic, and at any vertex where $2n$ edges meet, the angles of intersection are all $\pi/n$. Choose a $2$-coloring of $G$, and for each white (resp. black) face $F$ of $G$, we let $\phi_F$ denote a conformal mapping of $F$ onto $\mathbb{D}$ (resp. $\mathbb{D}^*$).
\end{notation}


We will now define a sequence of graphs $(G_n)_{n=1}^\infty$ by subdividing the edges in $G$. The graphs $G_n$ and $G$ will coincide as embedded subsets of the plane, but the diameters of the edges of $G_n$ will $\rightarrow0$ as $n\rightarrow\infty$. 

\begin{definition}\label{Gndefn} For each $n\in\mathbb{N}$, we define a graph $G_n$ as follows. Let $F$ be a white face of $G$. Define
\begin{equation} \tilde{V}_n^F:=\phi_F^{-1}\{\zeta: \zeta^n=\pm1\}, \end{equation}
and color each point in $\tilde{V}_n^F$ black or white according to whether the point is a pullback of an $n^{\textrm{th}}$ root of $+1$ or $-1$, respectively. For each vertex $v\in V(G)$ satisfying $v\in\partial F$, let $I_v$ be the component of $\partial F\setminus \tilde{V}_n^F$ containing $v$, and denote by $x_w$ (resp. $x_b$) the white (resp. black) endpoint of $I_v$. Let $I_w$ denote the component of $\partial F\setminus \tilde{V}_n^F$ such that $\overline{I_w}\cap \overline{I_v}=\{x_w\}$. Define $V_n^F$ by removing $x_w$, $x_b$ from $\tilde{V}_n^F$, and adding $v$ (colored black) as well as the midpoint of $I_w$ (colored white). Doing so over all vertices $v\in V(G)\cap\partial F$ defines $V_n^F$, and we set $V_n$ to be the union of $V_n^F$ over all white faces $F$. We define $G_n$ to be the graph obtained by subdividing $G$ at the vertices $V_n$, in other words $V(G_n)=V_n$.
 \end{definition}


We will now explain how new edges and vertices can be added to each black face $F$ of the graph $G_n$ so that there exists a small adjustment of $\phi_F$ which maps the edges of this modified black face (with added edges and vertices) to the $m^{\textrm{th}}$ roots of $\pm1$ for some $m\gg n$ (see Figure \ref{folding_picture}). First we will need to introduce several definitions.

\begin{definition}\label{length_mult_defn} Suppose $e$, $f$ are rectifiable Jordan arcs, and $h: e \rightarrow f$ is a homeomorphism. We say that $h$ is \emph{length-multiplying} on $e$ if  the push-forward (under $h$) of arc-length measure on $e$ coincides with the arc-length measure on $f$ multiplied by $\textrm{length}(f)/\textrm{length}(e)$.
\end{definition}

\begin{notation} For a graph $\mathcal{G}\subset\Chat$ and $C<\infty$, we let 
\begin{equation}\label{neighborhoodofgraph} N_C(\mathcal{G}):=\bigcup_{e} \{ z\in\Chat : \dist(z,e) < C\cdot\textrm{diam}(e) \},
\end{equation}
where the union in (\ref{neighborhoodofgraph}) is taken over all edges $e$ in $\mathcal{G}$. For $m\in\mathbb{N}$, we let \begin{equation} \nonumber \mathcal{Z}_m^{\pm}:=\{ z \in \mathbb{T} : z^m=\pm1 \} \textrm{ and }  \mathcal{Z}_m:=\mathcal{Z}_m^+\cup\mathcal{Z}_m^-. \end{equation}
\end{notation}

\begin{definition}\label{vertsuppdefn} Suppose $\Omega\subset\mathbb{C}$ is a domain such that $\partial\Omega\subset\mathcal{G}$ for some graph $\mathcal{G}$, and let $C<\infty$. We say a quasiregular mapping $g: \Omega \rightarrow \Chat$ is $(C, \mathcal{G})$-\emph{supported} if $\textrm{supp}(g_{\overline{z}})\subset N_C(\mathcal{G})$. We say a domain $W\subset\Omega$ is a \emph{$(C, \mathcal{G})$-tree domain in $\Omega$} if $W\subset N_C(\mathcal{G})$ and $\partial W$ consists of $\mathcal{G}\cap\partial\Omega$ together with a collection of pairwise disjoint trees rooted at the vertices of $\mathcal{G}$. 
\end{definition}


\begin{thm}\label{complicated_folding_adjustment} There exist $K$, $C<\infty$ so that for every $n\in\mathbb{N}$ and every black face $F$ of $G$, there exists a $(C, G_n)$-tree domain $W=W(n, F)$ in $F$, an integer $m=m(n, F)$, and a $K$-quasiconformal mapping $\psi=\psi(n, F): \mathbb{D}^*\setminus\phi_F(\partial W) \rightarrow \mathbb{D}^*$ so that: 
\begin{enumerate}
 \item $\psi\circ\phi_F: W \rightarrow \mathbb{D}^*$ is $(C, \partial W)$-supported, and $\psi\circ\phi_F(z)=\phi_F(z)$ off of $\textrm{supp}((\phi_F\circ\psi)_{\overline{z}})$, 
 \item for any edge $e$ of $\partial W\cap G_n$, $\psi\circ\phi_F(e)$ is a component of $\mathbb{T}\setminus\mathcal{Z}_m$, and $\psi\circ\phi_F$ is length-multiplying on $e$, and
 \item for any edge $e$ of $\partial W\setminus G_n$ and $x\in e$, the two limits $\lim_{W\ni z\rightarrow x}(\psi\circ\phi_F)^m(z)$ are conjugate points on $\mathbb{T}$. 
 \end{enumerate}
\end{thm}

\begin{figure}[ht!]
{\includegraphics[width=1\textwidth]{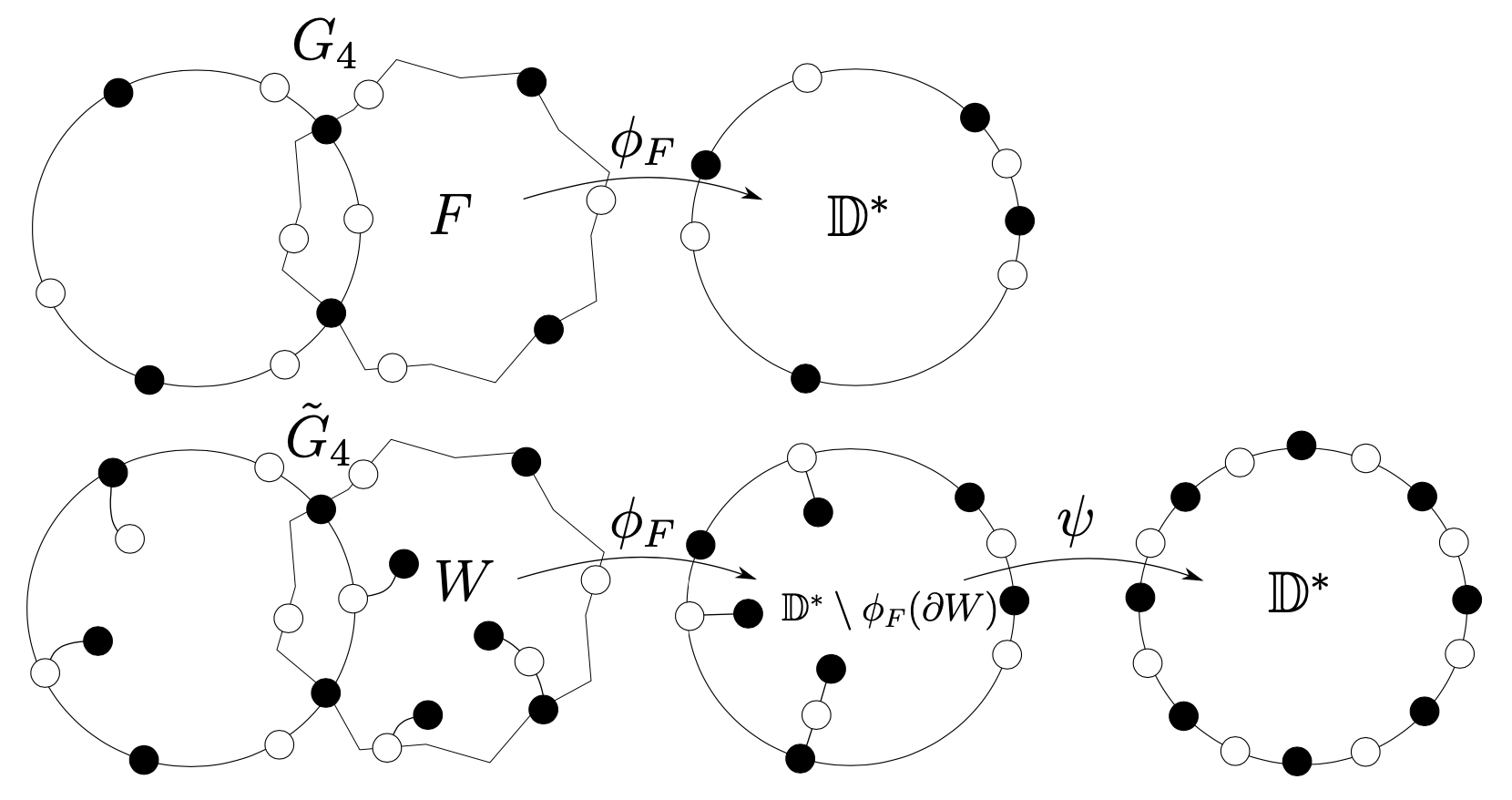}}
	\caption{Illustrated is Theorem \ref{complicated_folding_adjustment} and some of the notation in this section. Here $n=4$ and $m(4, F)=8$. The domain $W$ is obtained by removing the pictured trees from $F$. The trees are chosen such that the added vertices are mapped, by $\psi\circ\phi_F$, onto the $8^{\textrm{th}}$ roots of $\pm1$ as pictured. The map $\psi$ does not extend to a single-valued map on $\phi_F(\partial W\setminus \partial F)$, but for any $x\in \phi_F(\partial W\setminus \partial F)$, the two limits $\lim_{W\ni z\rightarrow x}\psi(z)$ are mapped to conjugate points on $\mathbb{T}$ by $z\mapsto z^8$..}
\label{folding_picture}
\end{figure}

\begin{rem} Theorem \ref{complicated_folding_adjustment} is essentially a summary of several technical lemmas in \cite{Bis15}; we will summarize the main idea in this Remark for the reader's convenience and refer to \cite{Bis15} or \cite{Bishop-Lazebnik} for the details (\cite{Bis15} contains the original proofs, and \cite{Bishop-Lazebnik} contains a simpler version of the argument which is sufficient for this paper). One may assume that $F=\mathbb{D}$, in which case $V_n\cap F$ consists of unevenly spaced points on $\mathbb{T}$. The main part of the proof of Theorem \ref{complicated_folding_adjustment} is to construct segments rooted at $V_n\cap F$ and a quasiconformal mapping which maps the union of $V_n\cap F$ with the vertices on the constructed segments onto the $m^{\textrm{th}}$ roots of $\pm1$. Neither the quasiconformal map $\psi$ nor the constructed segments are difficult to describe: the segments have as many vertices as are needed to ``fill in'' the missing   $m^{\textrm{th}}$ roots, and $\psi$ is piecewise-linear in logarithmic coordinates.  
\end{rem}

\begin{notation} We denote $\tilde{G}_n:=\bigcup_F \partial W(n, F)$, where the union is over all black faces $F$. In other words, $\tilde{G}_n$ is the graph $G_n$ together with the trees added in Theorem \ref{complicated_folding_adjustment} over all black faces. We let $F':=F\setminus\partial W$ for any black face $F$, in other words $F'$ is the black face $F$ after removing the trees added in Theorem \ref{complicated_folding_adjustment}. 
\end{notation}

\begin{lem}\label{adjustingphif} There exist $C$, $K<\infty$ so that for every $n\in\mathbb{N}$ and white face $F$, there is a $K$-quasiconformal, $(C, G_n)$-supported map $\tilde{\phi}_F: F \rightarrow \mathbb{D}$ so that
\begin{enumerate}
\item $\tilde{\phi}_F(V_n\cap\partial F)=\mathcal{Z}_n$, 
\item $\tilde{\phi}_F$ is length-multiplying on each edge of $G_n\cap\partial F$, and
\item $\tilde{\phi}_F(z)=\phi_F(z)$ off of $\textrm{supp}((\tilde{\phi}_F)_{\overline{z}})$. 
\end{enumerate}
\end{lem}

\begin{proof} Recall from Definition \ref{Gndefn} that there are $2\cdot|V(G)\cap F|$ many vertices which $\phi_F$ does not map onto an $n^{\textrm{th}}$ root of $\pm1$. We fix this by defining $\check{\phi}_F$ as a post-composition of $\phi_F$ with a self-homeomorphism of $\mathbb{D}$ which is the identity outside of a $1/n$-radius neighborhood of $\mathbb{T}$, and inside the $1/n$-radius neighborhood of $\mathbb{T}$ maps linearly the components of $\mathbb{T}\setminus \phi_F(V_n)$ onto the components of $\mathbb{T}\setminus \mathcal{Z}_n$. This can be done with quasiconformal constant independent of $n$. 

Next, consider the $2$-quasiconformal map
\begin{equation} 
\chi(z):= \begin{cases} 
z/|z|^{1/2} & \textrm{ if } |z|\leq 1, \\
z &\textrm{ if } |z|\geq1. \\	 
\end{cases} 
\end{equation}
Let $\chi_{r,\zeta}(z):=r\cdot\chi((z-\zeta)/r)+\zeta$ be a rescaled version of $\chi$. Let $D(z,r):=\{\zeta : |\zeta-z|<r\}$. Define
\begin{equation} 
\Phi_F:= \begin{cases} 
\chi_{1/n, \phi(v)} \circ \check{\phi}_F & \textrm{ in } \check{\phi}_F^{-1}(D(\check{\phi}_F(v), 1/n)) \textrm{ for any } v\in V(G) \\
\check{\phi}_F &\textrm{ otherwise.} \\	 
\end{cases} 
\end{equation}
The map $\Phi_F$ has the property that $|\Phi_F'|$ is bounded away from $0$ and $\infty$ for $\zeta\in\partial F$, with a bound independent of $n$. Let $\textrm{LHP}:=\{z : \textrm{Re}(z)<0\}$, and denote by
\begin{equation} \hat{\Phi}_F: \textrm{LHP} \rightarrow \textrm{LHP}
\end{equation}
a $2\pi i$-periodic lift of $\Phi_F$ under $2\pi i$-periodic universal covering maps of $\textrm{LHP}$ onto $F\setminus\{\phi_F^{-1}(0)\}$, $\mathbb{D}\setminus\{0\}$. Let $\xi: i\mathbb{R} \rightarrow i\mathbb{R}$ denote the piecewise-linear map which sends each component $I$ of $i\mathbb{R}\setminus\hat{\Phi}_F^{-1}(\{ \pi i/n : n \in\mathbb{Z}\})$ linearly onto $\hat{\Phi}_F(I)$, and let $c$ denote the maximum diameter of a component of  $i\mathbb{R}\setminus\hat{\Phi}_F^{-1}(\{ \pi i/n : n \in\mathbb{Z}\})$. Since $|\hat{\Phi}_F'|$ is bounded away from $0$ and $\infty$ on $i\mathbb{R}$, the linear interpolation between $\hat{\Phi}_F$ on $-c+i\mathbb{R}$ and $\xi$ on $i\mathbb{R}$ has dilatation bounded independently of $n$ (see, for instance, Lemma 2.1 of \cite{Bishop-Lazebnik}, or Theorem A.1 of \cite{MR4008367}). Define $\tilde{\Phi}_F$ to be $\hat{\Phi}_F$ for $\textrm{Re}(z)<-c$, and to be the aforementioned linear interpolation in $-c\leq\textrm{Re}(z)\leq0$. The projection (under the covering maps) of the map $\tilde{\Phi}_F$ back to a map $F \mapsto \mathbb{D}$ satisfies the conclusions of the Lemma. 
\end{proof}

\begin{definition}\label{h_n_defn} Recalling the map $\iota$ of Section \ref{interpolation_section}, we define, for every $n\in\mathbb{N}$, a function $h_n: \Chat\setminus \tilde{G}_n \rightarrow \Chat\setminus\mathbb{T}$ as follows: 
\begin{equation}\label{g_n_defn} h_n(z):= \begin{cases} 
\iota_n\circ\tilde{\phi}_F & \textrm{ in every white face } F \\
\iota_{m(n, F)} \circ \psi \circ \phi_F &\textrm{ in } F' \textrm{ for every black face } F \\	 
  \end{cases} 
\end{equation}
\end{definition}

\noindent The map $h_n$ does not extend to a single-valued function across the edges of $\tilde{G}_n\setminus G$. We will fix this using the following map.

\begin{definition}\label{sigma_definition} Let $\mu(z):=(z+1)/(z-1)$ be the conformal mapping of $\{z : |z|>1\}$ 
onto the right-half plane, taking the triple $(-1,1,\infty)$
	to $(0, \infty, 1)$. Note that $\mu^{-1} = \mu$.
Consider the $3$-quasiconformal map $\nu : \{\textrm{Re}(z)>0\} \to \mathbb{C}\setminus(-\infty,0]$ defined by
$$ 
\nu(r e^{i\theta}) = 
\begin{cases}
r e^{i\theta},& |\theta|\leq \pi/4,  \\
r e^{i3(\theta -\pi/4)+i \pi/4},&  \pi/4 < \theta < \pi/2  \\
r e^{i3(\theta +\pi/4)+i \pi/4},&  -\pi/4  >  \theta >- \pi/2. 
\end{cases}
$$ 
Thus $\sigma := \mu^{-1} \circ \nu \circ \mu = \mu \circ \nu \circ \mu$
is $3$-quasiconformal from $\{|z|>1\}$ onto $\mathbb{C}\setminus[-1,1]$ (see Figure \ref{fig:sigma_def}). 
\end{definition}

\begin{figure}[ht!]
{\includegraphics[width=1\textwidth]{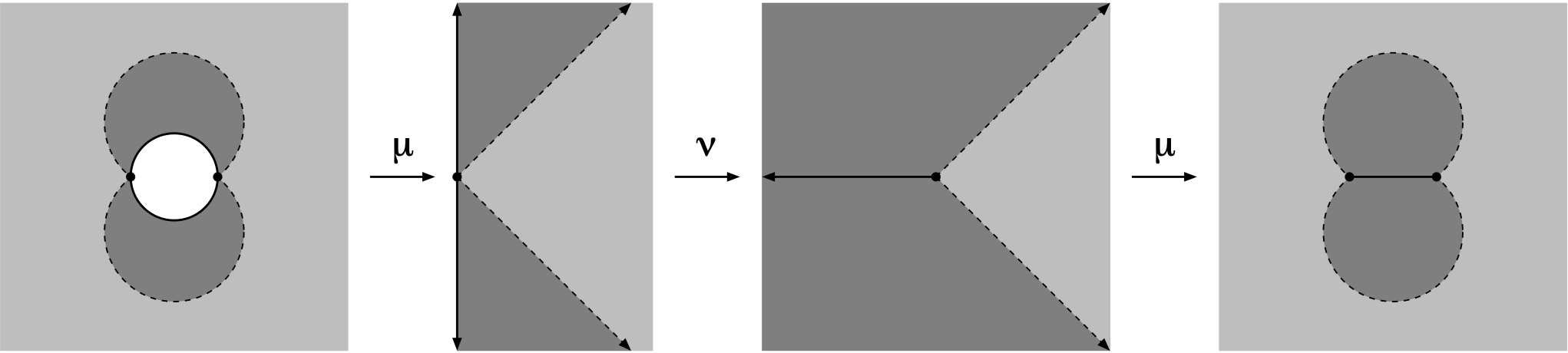}}
	\caption{A  quasiconformal map $\sigma: \{|z|>1\} \rightarrow \complex \setminus [-1,1]$. The  composition 
	$\mu \circ \nu \circ \mu$ is the identity 
	in the light gray region, and $3$-QC in the dark gray.}
\label{fig:sigma_def}
\end{figure}

\begin{definition}\label{g_n_defn} For every $n\in\mathbb{N}$, we define a map $g_n: \Chat\setminus \tilde{G}_n \rightarrow \Chat$ as follows. In each white face $F$, we set $g_n:=h_n$. If $F$ is a black face, we define $g_n$ by adjusting the definition of $h_n$ in $F'$ as follows. Let $X\subset \complex$ be the set where the map $\sigma$ of Definition \ref{sigma_definition}
is not conformal (this is the dark shaded region in 
the leftmost picture of Figure \ref{fig:sigma_def}). The set $(\iota_m)^{-1}(X)$ has $2m$ components in $\mathbb{D}^*$, one neighboring each component of $\mathbb{T}\setminus\mathcal{Z}_m$ with diameter comparable to $1/m$. 
Thus the set $(h_n|_{F'})^{-1}(X)$ has $2m(n, F)$ components in $F'$, one neighboring each side of each edge of 
$\tilde{G}_n\cap\partial F'$. We let $U$ denote the union of those 
components of $(h_n|_{F'})^{-1}(X)$ which neighbor an edge of 
$(\tilde{G}_n\cap\partial F')\setminus G_n$. We define
\begin{equation}\label{definition_of_g} 
	g_n(z):= \begin{cases} 
      \sigma \circ h_n(z) &\textrm{ if } z\in U \\
      h_n(z) & \textrm{ if }z\in F' \setminus U,
 \end{cases} \end{equation}
in $F'$ for each black face $F$.
\end{definition}

\begin{prop}\label{extension_prop} For every $n\in\mathbb{N}$ the function $g_n: \Chat\setminus \tilde{G}_n\rightarrow\Chat$ extends to a $K$-quasiregular function $g_n: \Chat \rightarrow \Chat$, where $K$ is independent of $n$, and 
\begin{equation}\label{areagoingto0} \emph{area}(\emph{supp}(g_n)_{\overline{z}})\xrightarrow{n\rightarrow\infty}0.
\end{equation}
\end{prop}

\begin{proof} To check that $g_n$ extends quasiregularly across $\tilde{G}_n$, it suffices (by a standard removability result for quasiregular mappings) to check that $g_n$ extends continuously across each edge of $\tilde{G}_n$. There are two types of edges to check: those belonging to $G_n$, and those belonging to $\tilde{G}_n\setminus G_n$. 

Let $e$ be an edge in $G_n$, and $g_n^B$ (resp. $g_n^W$) denote the boundary values of the map $g_n$ restricted to the black (resp. white)
 face containing $e$ on its boundary. Then, since $z\mapsto z^d$ is length-multiplying on $\mathbb{T}$ for every $d$, and $\iota_n(z)=z^n$ for $z\in\mathbb{T}$, Lemma \ref{adjustingphif} implies that $g_n^W(e)=\pm T\cap\mathbb{H}$ and $g_n^W$ is length-multiplying on $e$. Similarly, Theorem \ref{complicated_folding_adjustment}(2) implies that $g_n^B=\pm T\cap\mathbb{H}$ and $g_n^B$ is length-multiplying on $e$. Since $g_n^W$ and $g_n^B$ both map $e$ to the same set, agree on the endpoints of $e$, and both are length-multiplying on $e$, we conclude that $g_n^W$, $g_n^B$ agree pointwise on $e$. Hence $g_n$ extends continuously across $e$. 
 
Now, let $e$ be an edge in  $\tilde{G}_n\setminus G_n$. Then, by Theorem \ref{complicated_folding_adjustment}(3) and Definition \ref{h_n_defn}, the two limits $\lim_{F'\ni z\rightarrow x}h_n(z)$ are conjugate points of $\mathbb{T}$ for any $x\in e$. Thus, since $\sigma$ maps conjugate points of $\mathbb{T}$ onto the same point of $[-1,1]$, Definition \ref{g_n_defn} implies that the limit $\lim_{F'\ni z\rightarrow x}g_n(z)$ exists and is real-valued, and hence $g_n$ extends continuously across $e$. 

Thus we have shown that $g_n$ extends quasiregularly across $\tilde{G}_n$, and since the quasiconformality constants of the maps $\sigma$, $\iota$, $\psi$, $\tilde{\phi}_F$ in the definition of $g_n$ do not depend on $n$, neither does the quasiconformality constant of $g_n$. Lastly, the relation (\ref{areagoingto0}) follows since each of the maps in the definition of $g_n$ is either conformal or $(C, \tilde{G}_n)$-supported for $C$ independent of $n$. 
\end{proof}

\begin{notation} Recall Definition \ref{defnofcn} of the constant $c_n$. For $n\in\mathbb{N}$, we let
\begin{equation} A_n:= A\left(c_n^n+\delta c_n, 1/\max_F\left(c_{m(n, F)}^{m(n, F)}+\delta c_{m(n,F)}\right)\right), \end{equation}
where the maximum is over all black faces $F$. 
\end{notation}

\begin{prop}\label{location_of_cvs} For every $n$, we have 
\begin{equation}\label{locofcvseqtn} \{\pm1\}\subset\emph{CV}(g_n)\subset  A_n\textrm{ and } V(G)\subset \emph{CP}(g_n)
\end{equation}
\end{prop}

\begin{proof} If $F$ is a white face, then $g_n|_F:=\iota_n\circ\tilde{\phi}_F$. Since $\tilde{\phi}_F$ is quasiconformal, we have 
\begin{equation}\label{1e} \textrm{CV}(g_n|_F)=\textrm{CV}(\iota_n)\subset A_n
\end{equation}
by Lemma \ref{annulus_pullback_lemma}. Similarly, if $F$ is a black face, then $g_n|_{F'}$ is a composition of quasiconformal mappings with $\iota_{m(n,F)}$, and so again by Lemma \ref{annulus_pullback_lemma} we have:
\begin{equation}\label{2e} \textrm{CV}(g_n|_{F'})=\textrm{CV}(\iota_{m(n,F)})\subset A_n.
\end{equation}
The only points of $\tilde{G}_n$ where $g_n$ is locally $d:1$ for $d>1$ are the vertices of $\tilde{G}_n$ (which include $V(G)$), and the vertices of $\tilde{G}_n$ are mapped to $\pm1$ by $g_n$. This, together with (\ref{1e}) and (\ref{2e}), implies (\ref{locofcvseqtn}).
\end{proof}

\begin{notation} For $\varepsilon>0$, let $N_\varepsilon(E):=\{z: d(z,E)<\varepsilon\}$ denote the $\varepsilon$-neighborhood of a set $E\subset\mathbb{C}$. 
\end{notation}

\begin{prop}\label{pullback of annulus}  Let $\varepsilon>0$. Then, for all sufficiently large $n$, $g_n^{-1}(A_n)\subset N_{\varepsilon}(G)$. 
\end{prop}

\begin{proof} Let $F$ be a white face of $G$. Lemma \ref{annulus_pullback_lemma} implies that $(\iota_n|_\mathbb{D})^{-1}(A_n)\subset A(1-\varepsilon', 1)$ for any $\varepsilon'>0$ if $n$ is sufficiently large. Thus
\begin{equation}\label{whiteface} (g_n|_F)^{-1}(A_n)=\phi_F^{-1}((\iota_n|_F)^{-1}(A_n))\subset N_{\varepsilon}(G) \textrm{ for every white face } F \end{equation}
and sufficiently large $n$. Similarly, if $F$ is a black face of $G$, then $(\iota_{m(n,F)}|_{\mathbb{D}^*})^{-1}(A_n)\subset A(1, 1+\varepsilon')$ by Lemma \ref{annulus_pullback_lemma} and so the definition of $g_n|_{F'}$ implies that 
\begin{equation}\label{blackface} (g_n|_{F'})^{-1}(A_n)\subset N_{\varepsilon}(G) \textrm{ for every black face } F \end{equation}
and sufficiently large $n$. Together, (\ref{whiteface}) and (\ref{blackface}) imply the result. 
\end{proof}

\noindent Recalling $\textrm{CV}(g_n)\subset A_n$ by Proposition \ref{location_of_cvs}, we may define:

\begin{definition}\label{Sigma_n_defn} For every $n\in\mathbb{N}$, let $\Sigma_n\subset A_n$ be an analytic Jordan curve passing through the critical values of $g_n$. 
\end{definition} 

\noindent For each $n$ there are many analytic Jordan curves $\Sigma_n$ satisfying Definition \ref{Sigma_n_defn}; any of them will suffice in what follows. 

\begin{prop}\label{quasiregular_vers} Let $\varepsilon>0$. Then, for all sufficiently large $n$, the map $g_n$ satisfies 
\begin{equation}\label{hausdorff_ineq} d_H(g_n^{-1}(\Sigma_n), G)<\varepsilon. \end{equation}
\end{prop}

\begin{proof} The inequality (\ref{hausdorff_ineq}) will follow once we show that
\begin{equation}\label{first_ineq} g_n^{-1}(\Sigma_n)\subset N_{\varepsilon}(G) \textrm{, and }
\end{equation}
\begin{equation}\label{second_ineq} G \subset N_{\varepsilon}(g_n^{-1}(\Sigma_n)). 
\end{equation}
The inequality (\ref{first_ineq}) follows from Proposition \ref{pullback of annulus} since $\Sigma_n\subset A_n$. The inequality (\ref{second_ineq}) follows since $\Sigma_n$ passes through $+1$ by Proposition \ref{location_of_cvs}, and so $\iota_n^{-1}(\Sigma_n)\subset\mathbb{T}$ passes through all $n^{\textrm{th}}$ roots of $+1$ which fill a $\varepsilon'$-dense subset of $\mathbb{T}$ for any $\varepsilon'>0$ if $n$ is sufficiently large.
\end{proof}

\begin{definition} By Proposition \ref{extension_prop} and the Measurable Riemann Mapping Theorem, for each $n$ there exists a $K$-quasiconformal mapping $\phi_n: \Chat\rightarrow\Chat$ so that
\begin{equation*}\label{defining_eqtn}
r_n:=g_n\circ\phi_n^{-1}: \Chat\rightarrow\Chat 
\end{equation*}
is holomorphic (and hence a rational function). We normalize each $\phi_n$ so as to fix any three given points in $\Chat$.
\end{definition}

\begin{lem}\label{converging_to_id} 
The mappings $\phi_n$ converge to the identity uniformly on compact subsets of $\Chat$. 
\end{lem}

\begin{proof} 
This follows from (\ref{areagoingto0}) since the mappings $\phi_n$ are $K$-quasiconformal, all normalized to fix the same three points, and with $K$ independent of $n$ (by Proposition \ref{extension_prop}).
\end{proof}

\begin{prop}\label{rational_version} Let $\varepsilon>0$. Then, for all sufficiently large $n$, the rational map $r_n$ satisfies 
\begin{equation} d_H(r_n^{-1}(\Sigma_n), G)<\varepsilon. 
\end{equation}
\end{prop}

\begin{proof} We have that $r_n^{-1}(\Sigma_n)=\phi_n(g_n^{-1}(\Sigma_n))$, and so the conclusion follows from Proposition \ref{quasiregular_vers}, Lemma \ref{converging_to_id}, and the triangle inequality.
\end{proof}

\noindent \emph{Proof of Theorem \ref{density_thm}:} Since $\Sigma_n$ runs through the critical values of $g_n$, and we have $\textrm{CV}(g_n)=\textrm{CV}(r_n)$ and $\textrm{CP}(r_n)=\phi_n(\textrm{CP}(g_n))$, we have that   $r:=r_n$ and $\Gamma:=\Sigma_n$ satisfy the conclusions of Theorem \ref{density_thm} for all large $n$ by (\ref{locofcvseqtn}), Lemma \ref{converging_to_id} and Proposition \ref{rational_version}.  \qed

\appendix \section{}

In this appendix, we record the definition of a branched cover for the sake of completeness, and we provide several more examples of analytic nets in Figures \ref{third_ex}-\ref{fifth_ex}. The same conventions explained in the caption of Figure \ref{two_beginning_ex} hold for Figures \ref{third_ex}-\ref{fifth_ex}. 

\begin{definition}\label{TBC_defn} Let $d\geq 1$. A map $f: \Chat\rightarrow\Chat$ is called a \emph{branched covering of degree} $d$ if there is a finite subset $W\subset\Chat$ so that
\begin{enumerate}
\item $f: \Chat\setminus f^{-1}(W) \rightarrow \Chat\setminus W$ is a covering map of degree $d$, and
\item For every $w \in W$ and each $z\in f^{-1}(w)$, there exist neighborhoods $U$, $V$ of $z$, $w$ (respectively), an integer $n\geq 1$, and homeomorphisms $\phi_U: U \rightarrow\mathbb{D}$, $\phi_V: V\rightarrow \mathbb{D}$ so that $\phi_V \circ f\circ  \phi_U^{-1}(z) = z^n$ for all $z\in \mathbb{D}$.  
\end{enumerate}
We will call the smallest $W$ satisfying the above the \emph{critical values} of $f$, denoted $\textrm{CV}(f)$, and if $w\in\textrm{CV}(f)$ and $z\in f^{-1}(w)$ satisfies (2) with $n\geq2$, we call $z$ a \emph{critical point} of $f$, and we denote the set of critical points by $\textrm{CP}(f)$. 
\end{definition}

\begin{figure}
\centering
\begin{minipage}{.5\textwidth}
  \centering
  \includegraphics[width=.9\linewidth]{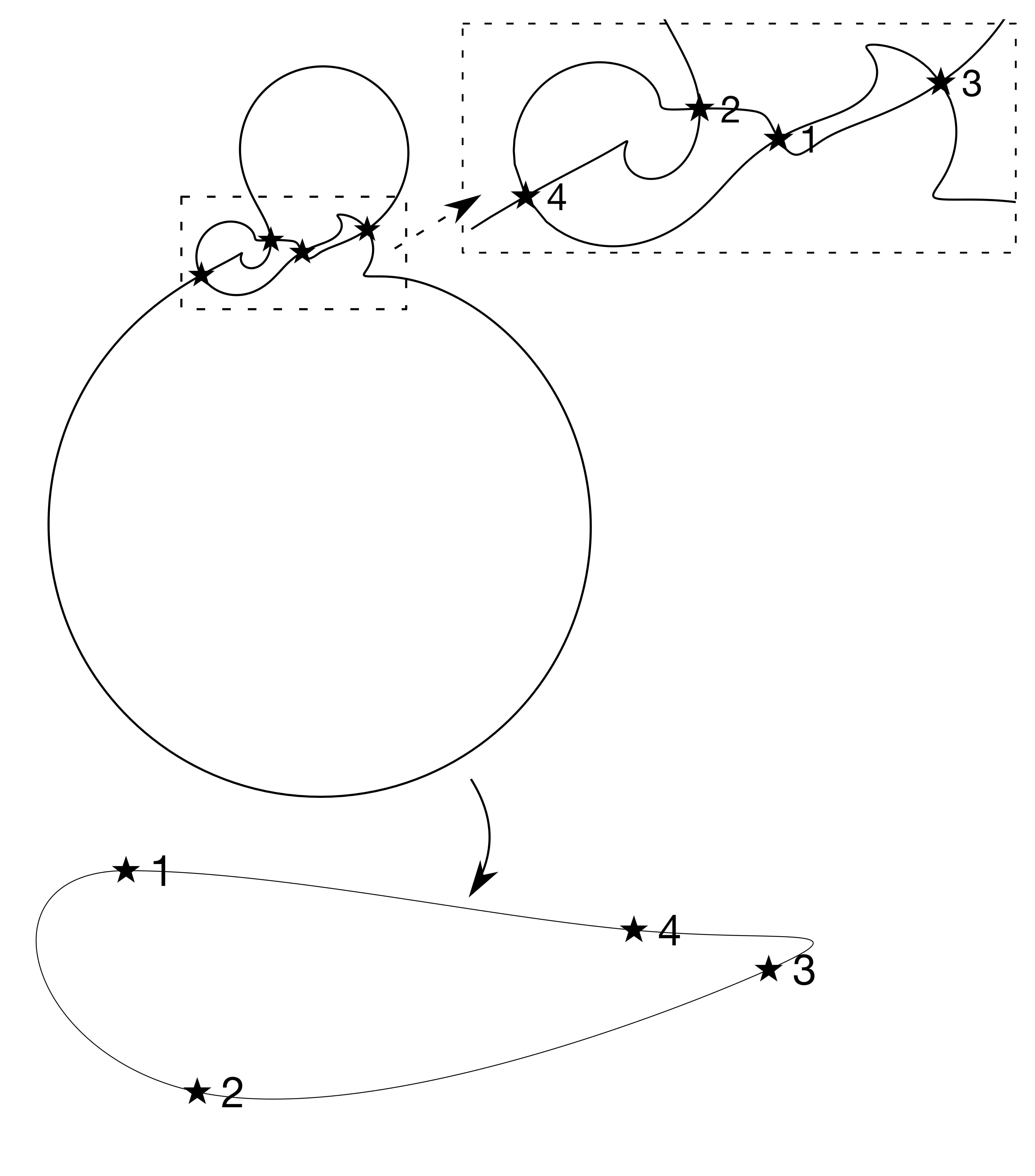}
  \captionof{figure}{A degree $3$ example.}
  \label{third_ex}
\end{minipage}%
\begin{minipage}{.5\textwidth}
  \centering
  \includegraphics[width=.9\linewidth]{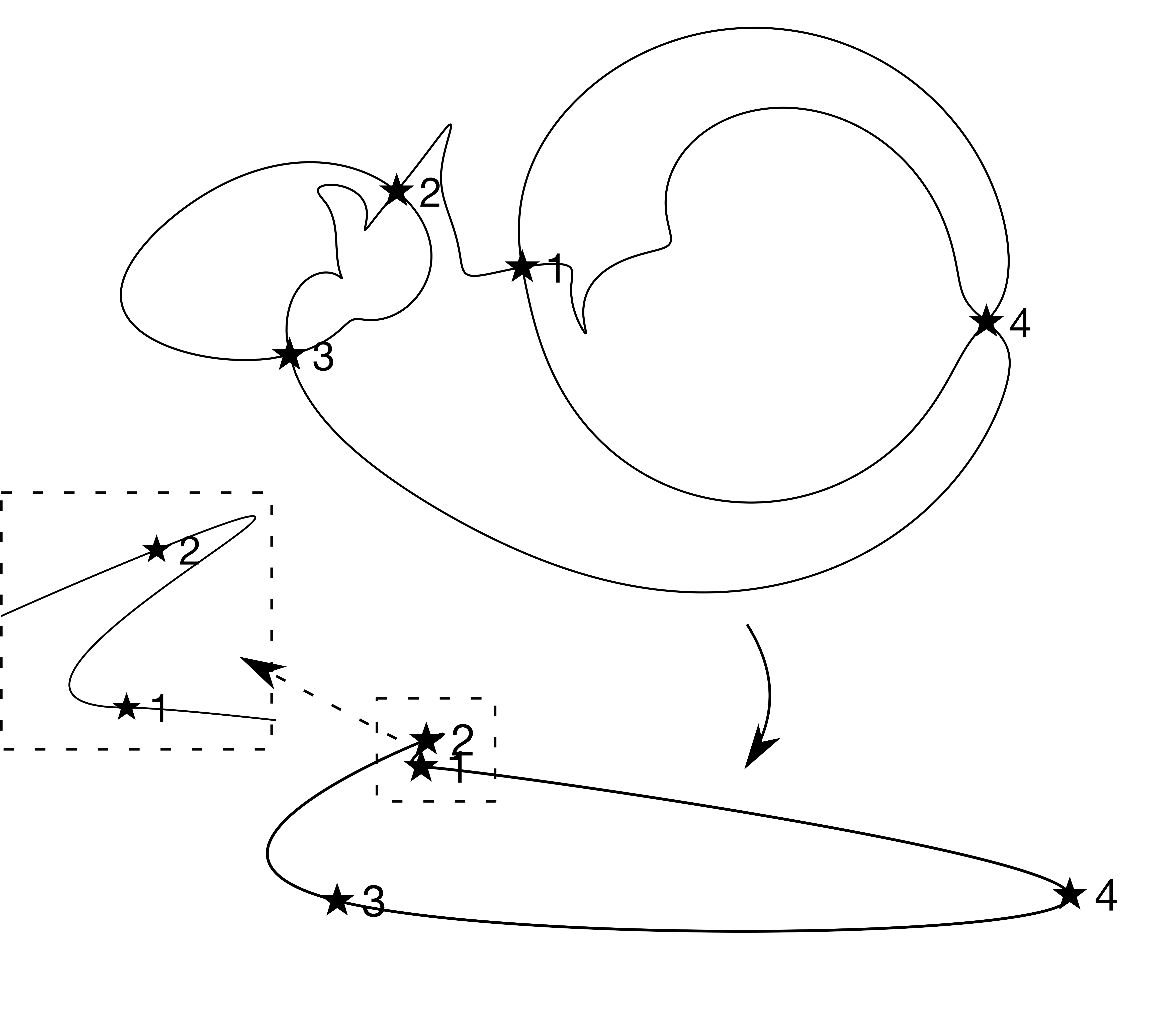}
  \captionof{figure}{Another degree $3$ example.}
  \label{fourth_ex}
\end{minipage}
\end{figure}

\begin{figure}
\centering
\includegraphics[width=1.1\linewidth]{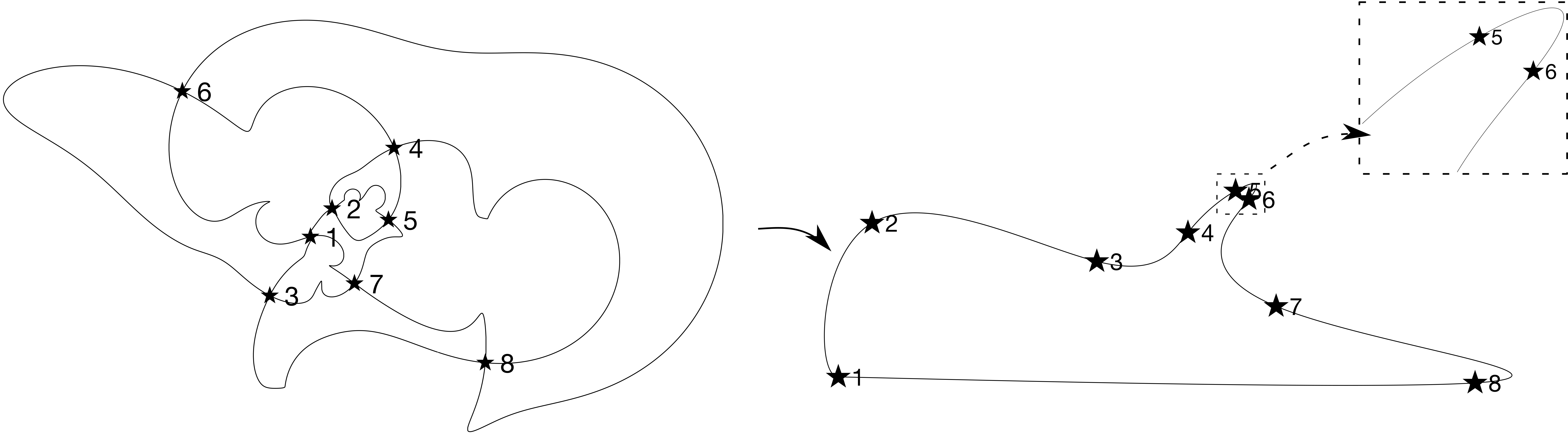}
\caption{A degree $5$ example.}
\label{fifth_ex}
\end{figure}


\bibliographystyle{alpha}

\end{document}